\documentclass[12pt,reqno,notitlepage]{amsart}
\usepackage{amsmath,amsfonts,amsthm,amssymb}
\usepackage{enumerate}
\usepackage{indentfirst}
\usepackage[dvips]{graphicx}
\usepackage[all]{xy}
\usepackage{stackrel}
\usepackage{marginnote}

\usepackage[centering, includeheadfoot, hmargin=1.0in, tmargin=1.0in, 
bmargin=1in, headheight=6pt]{geometry}
\usepackage[latin1]{inputenc}
\usepackage[T1]{fontenc}
\usepackage{verbatim}
\usepackage{color}
\usepackage[normalem]{ulem}

\usepackage{hyperref}  
\hypersetup{
pdfborder={0 0 0}, 
colorlinks=true, 
citecolor=blue,
linktoc=page,
pdfauthor={Thiago Fassarella and Nivaldo Medeiros}, 
pdftitle={Cremona monomial transformations and toric polar maps}
}
\renewcommand{\ref}{\hyperref}

\newcommand{\Hcal}{\mathcal{H}}

\newcommand{\Z}{\mathbb{Z}}
\renewcommand{\P}{\mathbb{P}}
\newcommand{\Pn}{\P^n}

\newcommand{\C}{\mathbb{C}}

\DeclareMathOperator{\euler}{\chi_\mathrm{top}} 

\newcommand{\csm}{c_{\mathrm{SM}}}

\newcommand{\abs}[1]{\left\lvert#1\right\rvert}

\newcommand{\one}{{\mathbf{1}}}

\newtheorem{theorem}{Theorem}[section]

\newtheorem{corollary}[theorem]{Corollary}

\newtheorem{thm}[theorem]{Theorem}

\newtheorem{prop}[theorem]{Proposition}

\theoremstyle{definition}

\newtheorem{example}[theorem]{Example}

\theoremstyle{plain}

\newcommand{\terminou}{\hfill$\lrcorner$}
 
\begin{document}
\author[{\scriptsize\rm{}T.~Fassarella and N.~Medeiros}]
{
T. Fassarella and N. Medeiros
}

\title[{\scriptsize\rm{Monomial transformations and toric polar maps}}]
{Monomial Cremona transformations and toric polar maps\\}

\begin{abstract}
Given a birational map in the three dimensional projective space defined by monomials of degree $d$, we prove that its inverse is defined by monomials of degree at most $d^2-d+1$.
\end{abstract}

\maketitle

%
%
\section{Introduction}
Let $\Pn$ denote the $n$-dimensional projective space over an arbitrary  field. 
A \emph{monomial Cremona transformation} of $\Pn$ is
a birational map $\varphi\colon \Pn \dashrightarrow \Pn$ whose components are monomials of the same positive degree $d$, assumed to have no common factors. A basic fact (see \cite[Theorem 4.2]{SV} or \cite[Theorem 2]{J}) is that its inverse $\varphi^{-1}$ is also
defined by monomials, say of degree $d'$.
We seek for an upper bound for $d'$ in terms of $d$ and $n$.

For $n=2$ and arbitrary $\varphi$, it is straightforward to see that $d'=d$ (cf. \S2 below), hence we assume $n\geq 3$.
Extensive computer calculations 
performed by Johnson in \cite{J}  led him to suggest the estimate 
\begin{equation}
\label{eq-Johnson-estimate}
d'\leq 1 + d-1 + (d-1)^2 + \dotsb + (d-1)^{n-1}.
\end{equation}
Moreover, up to permutation of variables or coordinates, the equality should occur only for the monomial map 
\begin{equation*}
\varphi_{n,d}(x_0:\cdots:x_n):=(x_0^d:x_0^{d-1}x_1:x_1^{d-1}x_2:\cdots:x_{n-1}^{d-1}x_n).
\end{equation*}
The case $d=2$ was settled in \cite[Theorem~2.6]{CS}.  
Johnson's 
computations show that the inequality holds for small values of $d$ when $n=3,4$. The answer is not known in general; we refer the reader to  \cite{J},  \cite{SV}, and \cite{DL} for more information and related results.
The aim of this short note is to prove  that Johnson's estimate holds for $n=3$, 
that is,
\begin{equation}
\label{eq-estimate-P3}
d'\leq d^2-d+1 
\end{equation} 
with equality attained only for $\varphi_{3,d}$. This is our Theorem~\ref{thm-bound}.

Our approach is to reformulate the problem in topological terms via the use of Chern-Schwartz-MacPherson (CSM) classes. We remark that the inverse of $\varphi$  is completely characterized by the exponents of the monomials, hence it is independent of the  ground field (\cite[Theorems 2.2 and 4.2]{SV}), so we may work over the complex numbers. Then, we consider the polynomial $f$ in $\C[x_0,\dotsc,x_3]$ given by the sum of all monomials defining $\varphi$. The key observation is that its \emph{toric polar map} $T_f$ and the monomial map $\varphi$ are given by the same linear system.
By using the results of
\cite{FMS}, we
relate the multidegrees of $\varphi$ with the CSM class of the surface $V(f)\subset \P^3$, which in turn enable us to prove Johnson's bound. 
In higher dimension the combinatorics is more intricate and we refrain to pursue the matter here.
To conclude, we point out that it would be interesting to compare the techniques used in this paper with 
the methods of \cite{Alu15}.

\section{Monomial transformations and toric polar maps}

Although all the concepts described below hold in arbitrary dimension, we specialize right from the start to the three dimensional complex projective space.

Let $A=(a_{ij})_{i,j=0,\dotsc,3}$ be a matrix of non-negative integers. Borrowing the terminology of \cite{J}, we say that $A$ is \emph{$d$-stochastic} if the sum of the entries in each row is $d$.
Note that the determinant of a $d$-stochastic matrix is  a multiple of $d$.
We associate to $A$ a rational monomial map
\begin{align*}
	\varphi\colon \P^{3} &\dashrightarrow \P^{3}\\
	x &\ \mapsto \  
	(x^{a_0}:\cdots:x^{a_3})
\end{align*}
where the vectors $a_0,\dotsc,a_3 \in \Z^4$ are  the rows of $A$ and $x^{a_i}=x_0^{a_{i0}}\dotsb x_3^{a_{i3}}$.
We  assume that the monomials have no common factor and thus each column of $A$ has at least one zero entry. 

The map $\varphi$ is birational if and only if 
 $\abs{\det(A)}=d$, see e.g. \cite[Proposition 2.1]{SV03} or
 \cite[Proposition~3.1]{GP03}. From now on we assume that is the case.
Let $d'$ be the common degree of the monomials (also assumed to have no common factor) defining the inverse $\varphi^{-1}$. 

Let $\Gamma\subset \P^3\times \P^3$ denote the graph of $\varphi$. Let $h$ (resp. $k$) be the pullback of the hyperplane class from the first (resp. second) factor. The class of this graph in the Chow ring of $ \P^3\times \P^3$ determines integers $d_0, \dotsc, d_3$ such that
\begin{equation*}
	\label{eq-graus-proj}
	[\Gamma] = d_0 k^3+d_1k^{2}h+ d_2kh^2 +d_3h^3 
	\qquad \in A_*(\P^3\times \P^3).
\end{equation*}
These integers are called the \emph{multidegrees} or {\it projective degrees} of the map ${\varphi}$.
If $\mathbb P^{3-j}\subset \P^3$ is a general linear  subspace of codimension $0\leq j \leq 3$, then
$d_j$ is the degree of the closed subset $\overline{\varphi^{-1}(\P^{3-j})}$. 

We always have $d_0=1$. We are assuming $\varphi$ birational, hence $d_3 = 1$. Finally, since the monomials defining $\varphi$ have no common {factor}, we 
have $d_1=d$.

Let $d_0',\dotsc,d_3'$ be the multidegrees  of $\varphi^{-1}$. Since the graph  of $\varphi^{-1}$ is the graph $\Gamma$ with the factors exchanged,
their multidegrees are reversed,  
that is,  
\[
(1,d_1',d_2',1) = (1,d_2,d_1,1).
\] 
In particular, we have $d'=d_1'=d_2$.
\medskip

Now we introduce the tools we shall employ to compute $d_2$. To the matrix $A$ we associate the polynomial
$f=\varphi^*(x_0+\dotsb+x_3) := \sum_{i=0}^3 x^{a_i}$ in $\C[x_0,\dotsc,x_3]$, given by sum of the monomials defining $\varphi$. Note that $f$ is homogeneous of degree $d$.
Let $D=V(f)\subset \P^3$ the surface given by the zero locus scheme of $f$. We consider the rational map 
\begin{align*}
T_f\colon \P^{3} &\dashrightarrow \P^{3}\\
x &\mapsto \big(x_{0}\frac{\partial f}{\partial x_0}(x):\cdots :x_{3}\frac{\partial f}{\partial x_3}(x)\big)
\end{align*}
called the \emph{toric polar map} of $f$. The key observation here is that both maps $\varphi$ and $T_f$ are given by the same linear
system: This follows from the fact that $A$, the exponent matrix of $f$, is invertible.
In particular, the monomial map $\varphi$ and the toric polar map $T_f$ have the same multidegrees and the same base locus.

The second ingredient we need are the 
\emph{Chern-Schwartz-MacPherson} classes.
These classes are defined for arbitrary varieties and, for smooth varieties, they agree with the  Chern class of the tangent bundle.
The degree of the CSM class of a
variety is its topological Euler characteristic, and the CSM class satisfy inclusion-exclusion; so, these classes may be seen as a generalization of the
Euler characteristic. We refer to \cite{AluCSMintro} for a gentle introduction;
for a quick summary of their basic properties, see
\cite[\S2.2]{AM09} or \cite[\S3]{FMS}.

Let $H_i=V(x_i)\subset\P^3$ for $i=0,\dotsc,3$ be the coordinate planes 
and $\Hcal=H_0\cup\cdots\cup H_3$  be their union.
The main result of \cite{FMS} is that the multidegrees of the toric polar map can be read directly from the CSM class of the open set $\P^3\setminus (D\cup \Hcal)$, provided that $D$ is `sufficiently general'.
Before stating it, let us setup some notation: For a constructible subset $Z\subseteq \P^3$, let $\one_Z$ be the constructible function yielding $1$ for $p\in Z$ and $0$ otherwise; in addition,
denote by $h$ the hyperplane class of $\P^3$. Here is the version needed for our purposes:

\begin{thm}[{\cite[Corollary 3.19]{FMS}}]
\label{thm-FMS}
Assume $f$ is nondegenerate with respect to its Newton polytope. Let $d_0,\dotsc,d_3$ be the multidegrees of the toric polar map $T_f$. Then:
\[
\csm(\one_{\P^3\setminus (D\cup\Hcal)}) = 
-d_3h^3 + d_2h^2 -d_1h + d_0 \qquad\in A_*\P^3.
\] 
\end{thm}
The nondegenerate condition means that for any face $\sigma$ of the Newton polytope of $f$, the surface $V(f_{\sigma})$ is smooth on the torus $(\C^*)^3$, where $f_\sigma$ is defined by the sum of monomials given by the vertices of $\sigma$; see e.g. \cite[\S3.4]{FMS} for details. As observed in \cite[Proposition~4.8]{FMS}, our polynomial $f$ is nondegenerate, since it comes from  an invertible monomial transformation. Then Theorem~\ref{thm-FMS} can be reformulated in terms of monomial maps:
\begin{corollary}
\label{thm-FMS*}
Let $\varphi\colon \P^3 \dashrightarrow \P^3$  be a monomial Cremona transformation. If $D=V(f)$ where  $f=\varphi^*(x_0+\dotsb+x_3)$, then 
\[
\csm(\one_{\P^3\setminus (D\cup\Hcal)}) = 
-h^3 + d'h^2 -dh + 1  \qquad\in A_*\P^3.
\] 
\end{corollary}
The corollary provides a way to express $d'$ in terms of the geometry of the surface $D$.

\section{Main theorem}

Let $\varphi\colon \P^3 \dashrightarrow \P^3$ be a birational map defined by monomials of the same positive degree $d$ and having no common factor.  Let $d'$ be the common degree of the monomials defining the inverse $\varphi^{-1}$. Our goal is to prove that \eqref{eq-estimate-P3} holds, with equality attained only for $\varphi_{3,d}$, up to permutation of variables or coordinates.   

Following the notation of the previous section,  we associate to $\varphi$ the surface $D=V(f)$, where $f$ is the polynomial given by the sum of the monomials defining $\varphi$.  We shall use Corollary~\ref{thm-FMS*} to compute $d'$; in order 
to do so, we use the inclusion-exclusion properties of the CSM class.
We have:
\begin{align*}
\csm(\one_{D\cup \Hcal}) & = \csm(\one_D) + \sum_{0\leq i \leq 3}\csm(\one_{H_i})
 -\sum_{0\leq i \leq 3}\csm(\one_{D\cap H_i}) 
 - \sum_{0\leq i<j \leq 3} \csm(\one_{H_i\cap H_j}) \\
&+  \sum_{0\leq i<j \leq 3}\csm(\one_{D\cap H_i\cap H_j})  + (\text{terms in codimension $\geq 3$}) 
 \qquad \qquad\qquad \in A_*\P^3,
\end{align*}
where we consider the pushforward of  all classes to the Chow group of the ambient space. We are interested in the coefficients in codimension 2, so the omitted terms are irrelevant. 

Let us describe each one of remaining classes in detail. We call the six  intersections $H_i\cap H_j$, with $0\leq i< j\leq 3$, the \emph{fundamental lines} of $\P^3$.

\begin{enumerate}[(a)]
\item Write $\csm(\one_D)=c_3h^3+c_2h^2+c_1h$. Then
$c_3$ is the topological Euler characteristic $\euler(D)$. Moreover, by \cite[Proposition 2.6]{Alu13} or \cite{Ohm03}, we have
\begin{equation*}
 \euler(D_h) = \int \csm(\one_{D_h})  = c_2 - c_1
\end{equation*}
where $D_h$ is a general plane section of $D$. 
Now, since the singular locus of $D$ is contained in the base locus of $T_f$
which, being equal to the base locus of $\varphi$, is supported on the six fundamental lines,  we see that $D_h$ is a plane curve of degree $d$ with only isolated singularities. Hence, from \cite[Example~14.1.5]{Ful84} we get
\begin{equation*}
	 \euler(D_h) = 3d-d^2 +   \mu(D_h)
 \end{equation*}
where $\mu(D_h) = \sum_{p\in D_h}\mu_p(D_h)$  is the sum of its Milnor numbers. Here we have used the fact that a smooth plane curve of degree $d$ has Euler characteristic $3d-d^2$. Since $c_1=\deg(D_\mathrm{red})=d$ we obtain
\begin{equation*}
\label{eq-c2}
c_2 = 4d-d^2 + \mu(D_h).
\end{equation*}
\item For the coordinate planes, the CSM class is just the pushforward of the Chern class of the tangent bundle of $\P^2$, so
for $i=0,\dotsc,3$:
\[
\csm(\one_{H_i}) = 3h^3+3h^2+h 
\qquad \in A_*\P^3.
\]
\item For $i=0,\dotsc,3$ we have
\[
\csm(\one_{D\cap H_i})=\euler(D\cap H_i)h^3 + k_ih^2
\qquad \in A_*\P^3
\]
where $k_i$ is the degree of the support of the plane curve $D\cap H_i$, that is,
\[
k_i:=\deg(D\cap H_i)_\mathrm{red}.
\]
\item For each fundamental line,
\[
\csm(\one_{H_i\cap H_j}) = 2h^3+h^2
\qquad \in A_*\P^3.
\] 
\item Finally, for $0\leq i< j\leq 3$,  define 
\[
\ell_{i,j}:=
\begin{cases}
1, & \text{if $D$ contains the line $H_i\cap H_j$}\\
0, & \text{otherwise}.
\end{cases}	
\]
Then we have:
\[
\csm(\one_{D\cap H_i\cap H_j}) = ah^3 +\ell_{i,j} h^2
\qquad \in A_*\P^3.
\]
where $a=2$ if $D$ contains the line $H_i\cap H_j$ and $a=\deg(D\cap H_i\cap H_j)_\mathrm{red}$, otherwise. 
\end{enumerate}
Now we can state a formula for $d'$ in
purely geometric terms.
\begin{prop}
\label{prop-eq-d'}
With notations as above, we have
\begin{equation*}
d' = d^2-4d-\mu(D_h)+\sum_{0\leq i \leq 3} k_i - \sum_{0\leq i < j \leq 3} \ell_{i,j}.
\end{equation*}
\end{prop}
\begin{proof}
By collecting the coefficients of $h^2$ in all expressions (a)--(e) above, using the fact that
$\csm(\one_{\P^3})=4h^3+6h^2+4h+1$ and substituting into Corollary~\ref{thm-FMS*},
we obtain
\begin{equation*}
	d'  = 6 - (4d-d^2 +\mu(D_h)) - 4\cdot 3  +\sum_{0\leq i \leq 3} k_i 
	+ 6\cdot 1 -  
	\sum_{0\leq i < j \leq 3} \ell_{i,j}
\end{equation*}
which yields the formula in the statement.
\end{proof}
We illustrate how our approach works in a significant example.
\begin{example}
\label{example-Johnson}
Johnson in \cite[Example~2]{J} shows that 
the maps $\varphi_{n,d}$, defined in the introduction, yield the equality in \eqref{eq-Johnson-estimate}. Let us give a different proof for $n=3$.
Given $d\geq 2$, consider the matrix
\[
A_d=
\begin{bmatrix}
	d &  &  &  \\
	d-1 & 1 & & \\
	& d-1 & 1 &\\
	& & d-1& 1 
\end{bmatrix}
\] 
which yields $\varphi_{3,d}$. 
Its base locus consists of the line
$x_0=x_2=0$ with the embedded point $(0:0:0:1)$ 
and the isolated point $(0:0:1:0)$.
Let $f_d = x_0^d+x_0^{d-1}x_1+x_1^{d-1}x_2 + x_2^{d-1} x_3$ be the associated polynomial. The surface $V(f_d)\subset\P^3$ is smooth
if $d=2$ and it has a unique singular point at $(0:0:0:1)$ when $d\geq 3$. Hence a general linear section is smooth, so the sum of its Milnor numbers is $0$. By letting $x_2=0$ we get
\[ 
f_d(x_0,x_1,0,x_3) = x_0^{d-1}(x_0+x_1)
\]
hence the reduced curve is a union of two lines and thus $k_2=2$. Moreover, we have $k_i=d$ for $i\neq 2$ and 
$\ell_{i,j}=0$ for all $i,j$ with exception of $\ell_{0,2}=1$, as it is easily checked. Hence, by Proposition~\ref{prop-eq-d'},
\[
d'=d^2-4d-0+(3d+2)-1=d^2-d+1
\]
agreeing with Johnson's result.
\terminou
\end{example}

In view of Proposition~\ref{prop-eq-d'}, 
we can rephrase  \eqref{eq-estimate-P3}
as
\begin{equation}
\label{eq-bound-d'}
\sum_{0\leq i \leq 3} k_i 
-\sum_{0\leq i < j \leq 3} \ell_{i,j} -\mu(D_h) 
\leq
3d+1.
\end{equation} 
A problem here is that  we do not have much control over the Milnor numbers. Fortunately that is not necessary, as a stronger inequality holds. This is our main result.
\begin{theorem}
\label{thm-bound}
With notations as above, for any birational monomial map $\varphi$ of $\P^3$ we have
\begin{equation}
\label{eq-bound-main-thm}
\sum_{0\leq i \leq 3} k_i 
-\sum_{0\leq i < j \leq 3} \ell_{i,j} 
\leq
3d+1
\end{equation}
and hence Johnson's estimate $d'\leq d^2-d+1$  holds.
Moreover, up to permutation of variables or coordinates, the equality
is attained only for the map $\varphi_{3,d}$.
\end{theorem}

\begin{proof}
We may assume $d\geq 2$. 
Let $A$ be the matrix associated to $\varphi$.
Its rows are
the exponents of the monomials, so that column $i$ collects the exponents for the variable $x_i$ in each monomial.
Each column has at least one zero. We may permute rows or columns of this matrix as we please, as this make no difference in the result. We use diagrams to indicate the shape of matrices; an `$\ast$' will denote a nonzero entry.

We have
\begin{equation*}
\label{eq-matriz-Cremona}
\sum_{j=0}^3 a_{ij} = d \quad \text{for }i=0,\dotsc,3  \qquad\text{and}\qquad \abs{\det(A)}=d,
\end{equation*}
as we already pointed out. These conditions are very restrictive. For example, note that $A$ cannot be written as block matrix such as 
\begin{equation*}
\label{eq-matriz-ruim}
\begin{bmatrix}
{B} & \mathbf{0} \\
\mathbf{0} & {C} 
\end{bmatrix}
\end{equation*}
since otherwise $B$ and $C$ would be 
$d$-stochastic, implying that $d^2$ divides $\det(A)$, which is not possible since $d\geq 2$. 
Similarly, $A$ cannot have two rows with a single nonzero entry.
\medskip

We are ready to go. Our analysis is divided in four cases.

\smallskip
\underline{Case I}. There is a column in $A$ with exactly one zero, say
\[
A=
\begin{bmatrix}
\ast & & & \\
\ast & & & \\
\ast & & & \\
0 & a & b & c 
\end{bmatrix}.
\]
Let $\delta$ denote the number of nonzero exponents in the set $\{a,b,c\}$.
Note that $\delta>0$.
We have $f(0,x_1,x_2,x_3)=x_1^ax_2^bx_3^c$,
so that $D\cap H_0$ is the union of $\delta$ lines. Hence $k_0=\delta$ and $\ell_{0,1}+\ell_{0,2}+\ell_{0,3}=\delta$. Since $k_i\leq d$ for arbitrary $i$, we get
\[
\sum k_i - \sum\ell_{i,j} \leq \delta+ 3d - \delta < 3d+1.
\]
Hence we have a strictly inequality in \eqref{eq-bound-main-thm}.
\smallskip

From now on we always assume that each column of $A$ has at least two zeroes. 

Our analysis is now guided by the geometry of the base locus of $\varphi$, which is supported on the six fundamental lines. We
remark that the base locus contains at least one line (\cite[Theorem~1]{DL}).

\smallskip
\underline{Case II}. The base locus contains two concurrent lines, say $\ell_{0,2}=\ell_{0,3}=1$. We may assume
\[
A=
\begin{bmatrix}
	a & c & 0 & 0  \\
	b & e & 0 & 0\\
	0 & r & \ast & \ast   \\
	0 & s & \ast & \ast 
\end{bmatrix}
\]
with $ab\neq 0$. We cannot have $ce\neq 0$ (otherwise $A$ would be a block matrix), so we may assume $c=0$ and thus $a=d$. Since 
the top left $2\times 2$ submatrix is $d$-stochastic, its determinant is a multiple of $d$, hence 
$e=1$ and $b=d-1$.
Then 
\[
f(x_0,x_1,0,x_3)=f(x_0,x_1,x_2,0)=x_0^{d-1}(x_0+x_1)
\]
so $k_2=k_3=2$ and thus
\[
\sum k_i - \sum\ell_{i,j} \leq 2d+2+2-2  < 3d+1
\]
where the rightmost inequality 
holds 
because $d\geq 2$. We have the strict inequality in \eqref{eq-bound-main-thm}.

\smallskip
\underline{Case III}. The base locus contains two skew lines, say $\ell_{0,3}=\ell_{1,2}=1$. Since $A$ cannot be a block matrix, we may assume
\[
A=
\begin{bmatrix}
	a & 0  & d-a &  0  \\
	b &  d-b  & 0   & 0 \\
	0 & 0 & d-c &  c \\
	0 & d-e & 0 & e 
\end{bmatrix}
\]
where the eight numbers are nonzero.
By permuting rows or columns, we may also
assume that $a$ is the minimum of these numbers and consequently that $d-a$ is the maximum. Here we have two possibilities to deal with:
\smallskip

\textbullet~{$c<e$}. Here we consider the restriction to the coordinate plane $x_0=0$:
\[
f(0,x_1,x_2,x_3) = x_3^c(x_2^{d-c}+x_3^{e-c}x_1^{d-e})
\]
{which implies} $k_0 \leq d-c+1$, and the restriction to the coordinate plane $x_1=0$ (note that $d-c\leq d-a$):
\[
f(x_0,0,x_2,x_3)= x_2^{d-c}(x_3^c+x_2^{c-a}x_0^a)
\]
{therefore}  $k_1 \leq   c+1$. Hence
\[
\sum k_i -\sum \ell_{i,j} \leq 
(d-c+1)+(c+1)+2d-2 < 3d+1
\]
yielding the strict inequality in \eqref{eq-bound-main-thm} as well.

\smallskip
\textbullet~{$c\geq e$}. 
Curiously enough, that cannot happen. To see this, note that we would have  $d-c \leq d-e$
and hence
\begin{align*}
	-\det(A) & = bc(d-a)(d-e)-ae(d-b)(d-c) \\
	& \geq be(d-a)(d-c) - ae(d-b)(d-c) \\
	& = e(d-c)(b-a)d.
\end{align*}
But $\abs{\det(A)}=d$, so we must have $e=d-c=b-a=1$. Since $a\leq e$, we get $a=1$, so $b=2$. This yields $\abs{\det(A)}=2(d-1)^3-(d-2)$, but this is a contradiction: This number is greater than $d$ whenever $d\geq 3$ and we cannot have $d=2$ because $b=2$. 

\smallskip
\underline{Case IV}. The base locus contains only one coordinate line, say 
$\ell_{0,2}=1$. We may assume
\[
A=
\begin{bmatrix}
	a & 0 & 0 & u \\
	b & c & 0 & v \\
	0 & e & s & 0  \\
	0 & r & t & w
\end{bmatrix}
\]
with $abst\neq 0$. 
We must have $er=uv=0$, otherwise the base locus would contain two lines. The case $e=u=0$ does not happen, otherwise $A$ would have two rows with a single nonzero entry.
So, we may assume
$e\neq0$ and $u=0$ (note the symmetry of the four corner $2\times 2$ blocks),
hence $r=0$ and $a=d$. Then we have 
$\det(A)=d(csw+vet)$ and thus 
\[
csw+vet=1.
\]
If $cw\neq 0$, then $v=0$, which implies $c=s=w=1$ and hence $b=e=t=d-1$, yielding the matrix $A_d$ of Example~\ref{example-Johnson}. On the other hand, if $cw=0$, then $v=e=t=1$, which implies 
$s=w=d-1$, hence $c=0$ and then $b=d-1$; again, up to a
permutation of rows or columns, we obtain 
 the matrix $A_d$ of Example~\ref{example-Johnson}. 
 
At any rate, in this case we have the equality in \eqref{eq-bound-main-thm}.
\medskip

Finally, if the equality holds in~\eqref{eq-bound-d'}, then the equality in \eqref{eq-bound-main-thm} holds as well.
By the above analysis, that happens only on the last case, which yields $\varphi_{3,d}$.
That finishes the proof of the theorem.
\end{proof}

%
%

\vspace{0.5cm}

\font\smallsc=cmcsc9
\font\smallsl=cmsl9

\noindent{\scriptsize\sc Universidade Federal Fluminense, Instituto de Matem\'atica e Estat\'istica.\\
Rua Alexandre Moura 8, S\~ao Domingos, 24210-200 Niter\'oi RJ,
Brazil.}

{\scriptsize\sl E-mail address: \small\verb?tfassarella@id.uff.br?}

{\scriptsize\sl E-mail address: \small\verb?nivaldomedeiros@id.uff.br?}

\end{document}